\documentclass[12pt,reqno]{amsart}
\usepackage{latexsym,amsmath}
\usepackage[left=3cm,top=2.5cm,right=3cm,bottom=2.5cm]{geometry}
\usepackage[usenames,dvipsnames]{color}
\usepackage{amsthm}
\usepackage{amssymb}
\usepackage{epsfig}
\usepackage{mathrsfs}
\usepackage{verbatim}
\usepackage[normalem]{ulem}
\DeclareMathOperator{\ex}{ex}

\newcommand{\hm}[1]{\leavevmode{\marginpar{\tiny%
$\hbox to 0mm{\hspace*{-0.5mm}$\leftarrow$\hss}%
\vcenter{\vrule depth 0.1mm height 0.1mm width \the\marginparwidth}%
\hbox to 0mm{\hss$\rightarrow$\hspace*{-0.5mm}}$\\\relax\raggedright #1}}}

\allowdisplaybreaks[2]
\newtheorem{theorem}{Theorem}[section]
\newtheorem{lemma}[theorem]{Lemma}

\theoremstyle{definition}

\newcommand\eps{\varepsilon}
\def\({\left(}
\def\){\right)}

\newcommand{\forb}{\mathrm{Forb}}

\DeclareMathOperator{\Fano}{Fano}

\begin{document}
\overfullrule=5pt

\newcommand{\eee}[1]{\textcolor{Red}{#1}}
\newcommand{\cca}[1]{\textcolor{Magenta}{#1}}
\newcommand{\ddd}[1]{\textcolor{Brown}{#1}}
\newcommand{\ccb}[1]{\textcolor{Green}{#1}}
\newcommand{\bbb}[1]{\textcolor{Blue}{#1}}
\newcommand{\ccc}[1]{\textcolor{Cyan}{#1}}
\newcommand{\ccd}[1]{\textcolor{Orange}{#1}}
\newcommand{\ccf}[1]{\textcolor{BrickRed}{#1}}
\newcommand{\ccg}[1]{\textcolor{NavyBlue}{#1}}
\newcommand{\yyy}[1]{\textcolor{ForestGreen}{#1}}
\newcommand\dju{\mathbin{\dot{\cup}}}

\title{A note on a stability result for the Fano plane}

\author[C. Hoppen]{Carlos Hoppen}
\address{Instituto de Matem\'atica e Estat\'{i}stica, UFRGS -- Avenida Bento Gon\c{c}alves, 9500, 91501--970 Porto Alegre, RS, Brazil}
\email{choppen@ufrgs.br}

\author[H. Lefmann]{Hanno Lefmann}
\address{Fakult\"at f\"ur Informatik, Technische Universit\"at Chemnitz,
Stra\ss{}e der Nationen 62, D-09107 Chemnitz, Germany}
\email{lefmann@informatik.tu-chemnitz.de}

\author[K. Odermann]{Knut Odermann}
\address{Fakult\"at f\"ur Informatik, Technische Universit\"at Chemnitz,
Stra\ss{}e der Nationen 62, D-09107 Chemnitz, Germany}
\email{knut.odermann@Informatik.TU-Chemnitz.de}

\thanks{This work has been produced as part of a binational program funded by Coordena\c{c}\~{a}o de Aperfei\c{c}oamento de Pessoal de N\'{i}vel Superior (CAPES) and DAAD via PROBRAL (CAPES Proc.~88881.143993/2017-01 and DAAD~57391132 and~57518130).}

\begin{abstract}
In this note, we adapt the proof of the (Tur\'{a}n) Stability Theorem for the Fano plane~\cite[Theorem~1.2]{keevash2005} to find an explicit dependency between the para\-meters $\eps$ and $\delta$. This is useful in the solution of a multicolored version for hypergraphs of an extremal problem about edge-colorings, known as the Erd\H{o}s-Rothschild problem, which may be considered for the Fano plane.
 \end{abstract}

\maketitle

\section{Introduction}

This note deals with stability in uniform hypergraphs. As usual, for an integer $r \geq 2$, an $r$-uniform hypergraph $H=(V,E)$  is a pair consisting of a \emph{vertex set} $V$ and of a set $E \subseteq \binom{V}{r}$ of \emph{hyperedges}, where $\binom{V}{r}$ is the set of all subsets of $V$ with cardinality $r$. For a fixed $r$-uniform hypergraph $F$, we say that a hypergraph $H=(V,E)$ is \emph{$F$-free} if it does not contain a copy of $F$ as a subhypergraph. Let $\forb_{F}(n)$ denote the family of all labeled $F$-free $r$-uniform hypergraphs on $n$ vertices.

Given an $r$-uniform hypergraph $F$, the \emph{hypergraph Tur\'{a}n problem} for $F$ consists of determining the \emph{Tur\'{a}n number} $\ex(n,F)=\max \{ |E(H)| \colon H \in \forb_{F}(n) \}$ and of characterizing those hypergraphs $H \in \forb_{F}(n)$ such that $|E(H)|=\ex(n,F)$, which are said to be \emph{$F$-extremal hypergraphs}.

It is often the case that the extremal hypergraph is unique up to isomorphism, and that any $F$-free configuration $H=(V,E)$ such that $|E|$ is close to the maximum size must be structurally similar to the extremal configuration for $F$. This is known as \emph{stability}, a concept that was observed in the graph-theoretical setting by Erd\H{o}s and Simonovits~\cite{simonovits} and that may be naturally extended to the hypergraph Tur\'{a}n problem. We should also mention that, compared with the graph case, much less is known about the Tur\'{a}n problem for hypergraphs. For more information about this problem and stability results for hypergraphs, we refer the reader to surveys by Keevash~\cite{Keevash2011} and by Mubayi and Verstra\"{e}te~\cite{MV16}.

The \emph{Fano plane}, hereafter denoted by $\Fano$, is the single linear 3-uniform hypergraph on seven vertices and seven hyperedges. (A hypergraph $H=(V,E)$ is \emph{linear} if $|e \cap f| \leq 1$ for all hyperedges $e,f \in E$, $e \neq f$.) It is the projective plane over the field with two elements. S\'{o}s~\cite{sos76} conjectured that the unique extremal hypergraph for the Fano plane is the balanced, complete, bipartite $3$-uniform hypergraph $B_n$ on $n$ vertices. This conjecture was verified for sufficiently large $n$ by Keevash and Sudakov~\cite{keevash2005} and, independently, by F\"uredi and Simonovits~\cite{FSstability}. An earlier paper of De Caen and F\"{u}redi~\cite{decaen_fueredi00} already established that 
$$\lim_{n \rightarrow \infty} \frac{\ex(n,\Fano)}{\binom{n}{3}}=\frac{3}{4}.$$
Recently, Bellmann and Reiher~\cite{BRstability} proved that the conjecture holds for every $n \geq 8$,  which is best possible. This means that, for $n \geq 8$,
 $$
 \ex(n,\Fano) = \binom{ \lfloor \frac{n}{2}  \rfloor}{2} \cdot \left\lceil \frac{n}{2}  \right\rceil +  \binom{ \lceil \frac{n}{2}  \rceil}{2}  \cdot \left\lfloor \frac{n}{2}  \right\rfloor.
 $$
 
In the papers of Keevash and Sudakov and of F\"uredi and Simonovits, the extremality of $B_n$ was obtained using stability. The precise result proved in~\cite[Theorem 1.2]{keevash2005} is the following. For a hypergraph $H= (V,E)$ and a subset $A \subseteq V$, we write $e_H(A)$ for the number of hyperedges $e$ of $H$ such that $e \subseteq A$. 
 \begin{theorem}
For all $\eps>0$, there are $\delta > 0$ and $n_0$ with the following property. If $H = (V,E)$ is a $\Fano$-free $3$-uniform hypergraph on $n \geq n_0$ vertices with at least $(1-\delta)\frac{3}{4} \binom{n}{3}$ hyperedges, then there is a partition $V(H) = A \cup B$ with the property that $e_H(A) + e_H(B) < \eps n^3$.
\end{theorem}
In this note, we wish to compute the dependency between $\delta$ and $\eps$ explicitly. Our motivation was to compute explicit bounds for an Erd\H{o}s-Rothschild-type problem involving the Fano plane~\cite{ERFano}, where we apply a stability method. 

With a careful analysis of the proof of~\cite[Theorem 1.2]{keevash2005} and performing some modifications to tighten the bounds, we obtain the following result.
\begin{theorem}\label{main}
For any fixed $0 < \delta <  (1/36)^8$,  there exists $n_0$ such that the following holds for all $n \geq n_0$. If $H = (V,E)$ is a $\Fano$-free $3$-uniform hypergraph on $n \geq n_0$ vertices with $\ex(n, \Fano)- \delta n^3/8$ hyperedges, then there is a partition $V(H) = A \cup B$ so that $e_H(A) + e_H(B)\leq (9/16) (3^2\cdot2^4 \cdot 139)^{1/8} n^3 < 1.94  \delta^{1/64}n^3 $.
\end{theorem}


\section{Preliminaries}

In this section, we state some results that will be useful in our proof of Theorem~\ref{main}. Several of the lemmas come from the paper of Keevash and Sudakov~\cite{keevash2005}. 

Let  $K_4^{(3)}$ be the complete $3$-uniform hypergraph on four vertices, which is sometimes called tetrahedron. The Tur\'{a}n number $\ex(n,K_4^{(3)})$ is still not known, and we shall use the following upper bound. 
\begin{lemma}[\bf{Chung and Lu~\cite{chung_lu99}}]\label{lemma:lemma_2.2}
Let $H$ be a $3$-uniform hypergraph on $n\geq n_0$ vertices with at least $((3 + \sqrt{17})/12 + o(1) ) \binom{n}{3} \approx 0.594 \binom{n}{3}$ hyperedges. Then $H$ contains a copy of $K_4^{(3)}$.
\end{lemma}
Numerical computations associated with a semidefinite program (see Razborov~\cite{razborov2010} and~\cite{FV2013}) lead to a better upper bound, but Chung and Lu's bound suffices for our purposes.

We shall also consider $K^{(3)}(2,2,2)$, the $3$-uniform hypergraph with six vertices whose vertex set may be partitioned as $V_1 \cup V_2 \cup V_3$, where $|V_i|=2$ for every $i \in [3]$, and whose edge set is given by all triples $e$ such that $|e \cap V_i|=1$ for every $i \in [3]$. Note that $K^{(3)}(2,2,2)$ contains eight edges. This hypergraph is sometimes called octahedron. Moreover, let $K_{\ell}$ be the complete graph on $\ell$ vertices.

For a $3$-uniform hypergraph $H=(V,E)$ and a vertex $x \in V$, whose neighborhood is denoted by $N_H(x)$, the \emph{link graph} $L(x)=(V',E')$ has vertex set $V'=N_H(x)$ and edge set $E'=\{e-x \colon x \in e \mbox{ and }  e\in E\}$. Of course, the number of edges of $L(x)$ is the degree of vertex $x$ in $H$. It is convenient to consider several link graphs
simultaneously, regarding them as a multigraph. This is a loopless graph in which each edge has some non-negative integral multiplicity. If $S \subseteq V$ is
a set of vertices of $H$, then the link multigraph $L(S)$ of $S$ is the multigraph given by the union of the link graphs of each vertex in $S$. Let $G(S)$ be the subgraph of the link multigraph $L(S)$ obtained by removing $S$.

\begin{lemma}
\label{lemma:observation_2.1}
Let $H=(V,E)$ be a $3$-uniform hypergraph.
\begin{itemize} 
\item [(i)] Let $\{x_1, x_2, x_3\}\in E$ be a hyperedge in $H$ and let $L(x_i)$ be the link graph of vertex $x_i$. If $G(\{ x_1, x_2, x_3 \})$ contains four vertices spanning a complete graph $K_4$  whose edge set can be partitioned into three matchings $M_1, M_2, M_3$ with $M_i \subseteq L(x_i)$ for all $i=1,2,3$, then $H$ contains a Fano plane.

\item[(ii)]  If there is a vertex $x$ whose link graph $L(x)$ contains three pairwise vertex-disjoint edges $e_1$, $e_2$, $e_3$ such that all triples $\{x_1,x_2,x_3\}$, with $x_i \in e_i$ for $i \in \{1,2,3\}$, are hyperedges in $H$, then $H$ contains a Fano plane.
\end{itemize}
\end{lemma}

\begin{lemma} [\bf{F\"uredi and K\"undgen~\cite{fueredi_kuendgen02}}]\label{lemma:lemma_2.3}
Every loopless multigraph  on $n$ vertices, where every four vertices span at most $20$ edges, has at most $3 \binom{n}{2} + n-2$ edges.
\end{lemma}

\begin{lemma} [\bf{De Caen and F\"uredi~\cite{decaen_fueredi00}}]\label{lemma:lemma_2.4}
Let $H$ be a $3$-uniform hypergraph which contains a copy of $K_4^{(3)}$ with vertex set $S$. Let $L(S)$ be the link multigraph of $S$ and let $G(S)$ be the multigraph obtained from $L(S)$ by deleting $S$. If the subgraph $G(S)$ contains a set of four vertices spanning at least $21$ edges, then $H$ contains a Fano plane.
\end{lemma}

Parts of the following lemma are from~\cite{keevash2005}.
\begin{lemma} \label{lemma:lemma_2.5}
Let $H=(V,E)$ be a $3$-uniform hypergraph which contains a copy of $K_4^{(3)}$ with vertex set $S$. Let $E'$ be the set of edges in $G(S)$ of multiplicity at least $3$.
\begin{itemize}
\item[(i)] If $E'$ contains a copy of $K_4$ that has at least one edge of multiplicity $4$, then $H$ contains a Fano plane.
\item[(ii)] If $E'$ contains a copy of $K_5$, then $H$ contains a Fano plane.
\item[(iii)] If $G(S)$ contains a copy of $K_4$ with vertex set $\{y_1, y_2, y_3, y_4 \}$, where the edges 
$\{y_1, y_2 \}$ and $\{y_1, y_3 \}$ both have multiplicity $4$, and edge $\{ y_2, y_3\}$ has multiplicity at least $3$, and the other three edges have in some order multiplicities at least $4,3,2$, then $H$ contains a Fano plane.
\item[(iv)]  If $G(S)$ contains a copy of $K_4$ with vertex set $\{y_1, y_2, y_3, y_4 \}$, where  edge 
$\{y_1, y_2 \}$ has multiplicity $4$ and edge $\{y_3, y_4 \}$ has multiplicity at least $2$, and edge $\{y_1, y_3 \}$ has multiplicity at least $2$ and edge $\{y_2, y_3 \}$ has multiplicity $4$, and the remaining edges have multiplicities at least $3$, 
then $H$ contains a Fano plane.
\item[(v)]  If $G(S)$ contains a copy of $K_4$ with vertex set $\{y_1, y_2, y_3, y_4 \}$, where  edges 
$\{y_1, y_2 \}$ and $\{y_1, y_3 \}$ have multiplicities $4$ and edge $\{y_2, y_3 \}$ has multiplicity at least $3$, and edge  $\{y_1, y_4 \}$ has multiplicity $4$ and edge $\{y_2, y_4\}$ has multiplicity at least $2$  and edge $\{y_3, y_4 \}$ has multiplicity at least $1$, 
then $H$ contains a Fano plane.
\end{itemize}
\end{lemma}

\begin{proof} We prove parts (i), (iv) and (v), the others are similar. 
 For part (i), partition the edge set of the copy of $K_4$ into three matchings $M_1,M_2,M_3$ of size two and assume that $M_3$ contains an edge with multiplicity 4. Since the edges of $M_1$ and $M_2$ have  multiplicity at least three, we may find distinct vertices $x_1,x_2 \in S$ such that the edges of $M_1$ lie in $L(x_1)$ and the edges of $M_2$ lie in $L(x_2)$. The edge with least multiplicity in $M_3$ must contain a vertex $x_3 \notin \{x_1,x_2\}$. The other edge has multiplicity 4 and hence contains $x_3$. We apply Lemma~\ref{lemma:observation_2.1}~(i) to get a copy of a  $\Fano$ plane.

For part~(iv), let $S=\{u_1,u_2,u_3,u_4\}$  and let $L(u_i)$ be the link graph for $u_i$, $i = 1,2,3,4$. Partition the edge set of the copy of $K_4$  with vertex set $\{y_1,y_2,y_3,y_4\}$ into three edge-disjoint matchings $M_1, M_2, M_3$. Let $J_i$ be the indices of the link graphs that contain both edges of $M_i$. Then, in some order, we have the following lower bounds on the sizes of the $J_i$: $4+2 - 4 = 2$, and $2+3-4=1$, and  $4+3-4 = 3$. The result follows by Lemma~\ref{lemma:observation_2.1}~(i) and Hall's theorem.

The proof of part~(v) is similar to that of part~(iv). Namely, let $S=\{u_1,u_2,u_3,u_4\}$  and let $L(u_i)$ be the link graph for $u_i$, $i = 1,2,3,4$. Partition the edges of the copy of $K_4$ on $\{y_1,y_2,y_3,y_4\}$ into three edge-disjoint matchings $M_1, M_2, M_3$. Let $J_i$ be the indices of the link graphs that contain both edges of $M_i$. Then, in some order, we have the following lower bounds on the sizes of the $J_i$: $4+1 - 4 = 1$, and $2+4-4=2$, and  $4+3-4 = 3$. The result follows by Lemma~\ref{lemma:observation_2.1}~(i) and Hall's theorem.
\end{proof}

\begin{lemma} \label{lemma:lemma_2.6}
For every fixed $0 < \alpha < 1/6$, there is a constant $n_0$ such that the following holds for every integer $n \geq n_0$. Let $H=(V,E)$ be a $3$-uniform hypergraph with $|V|=n$ and $|E| \geq \alpha n^3$. Then, the number of copies of $K^{(3)}(2,2,2)$ in $H$ is at least 
$$ c' \cdot \alpha^8  n^6 = \frac{3^{7}}{2^{11}} \cdot \alpha^8 n^6.
 $$
\end{lemma}

\begin{proof}
Let $\alpha<1/6$ and $H$ as in the statement of the lemma, fix $H_0=H$ and consider the following deletion process starting with $j=0$. While there is a vertex $v_i \in V(H_j)$ whose degree satisfies $d_{H_j}(v_i) \leq n^{3/2}$, define $H_{j+1}=H_j-v_i$. Let $H'=(V',E')$ be the hypergraph obtained at the end of this process, where $V' = \{v_1, \ldots , v_{n'} \}$. Assume that $xn$ vertices have been deleted upon reaching termination, where $0 \leq x\leq 1$. Then we must have 

$$
 xn^{\frac{5}{2}}  +   \binom{(1-x)n}{3} \geq \alpha n^3,
 $$
which implies, for $n \geq 4/\alpha^2$,
$$
\frac{(1-x)^3n^3}{6} \geq  \frac{\alpha n^3}{2},
 $$
 thus $x\leq 1 - (3\alpha)^{1/3}$. So $H' = (V',E')$ has $|V'|= n'\geq (3\alpha)^{1/3}n$ vertices and at least
$$
\alpha n^3 - ( 1 - (3\alpha)^{\frac{1}{3}}) n^{\frac{5}{2}} \geq  \frac{\alpha n^3}{2}
$$
hyperedges, for $n \geq 4/\alpha^2$. Let $\alpha' = \alpha /2$, hence $|E'| \geq \alpha' n^3$. 

For any real number $y > 1$, set $\binom{y}{2} = \frac{y(y-1)}{2}$. For $y \geq 2a$, we will use the inequality
\begin{eqnarray} \label{realbinomial}
a \cdot \binom{\frac{y}{a}}{2} \geq \frac{y^2}{4a}.
\end{eqnarray}

We shall give a lower bound on the number of $4$-cycles in the link graphs $L(v)$ for all $v \in V'$.
Consider the link graph $L(v_i)$ of vertex $v_i\in V'$ in $H'$, so that $L(v_i)$ contains $d_i= \deg_{H'}(v_i)$ edges, $i = 1, \ldots , n'$. 
Let $d_i(v_j)$ be the degree of vertex $v_j \in V'$ in the link graph $L(v_i)$, $j = 1, \ldots , n'$, where, for convenience, we assume that $v_i$ lies in $L(v_i)$ with $d_i(v_i)= 0$. 
The number of (unordered) pairs of edges incident with vertex $v_j$ in $L(v_i)$, that is the number of pairs of triples $\{ v_i, v_j , v\} \in  E'$,
is $\binom{d_i(v_j)}{2}$. Summing over all $j$ and using the convexity of the function $f(x) = \binom{x}{2}$,
i.e., $\sum_{i=1}^\ell \binom{x_i}{2} \geq \ell \cdot \binom{\sum_{i=1}^\ell x_i/\ell}{2}$, the total number of pairs of edges in $L(v_i)$  incident with some vertex  in $V(L(v_i))$ is
\begin{eqnarray} \label{eq:pairs1}
\sum_{j=1}^{n'}  \binom{d_i(v_j)}{2} &\geq& n' \cdot \binom{\frac{\sum_{j=1}^{n'} d_i(v_j)}{n'}}{2} = n'\cdot \binom{\frac{2d_i}{n'}}{2} \geq  n\cdot \binom{\frac{2d_i}{n}}{2} \stackrel{(d_i \geq n; \, \eqref{realbinomial})}{\geq} \frac{d_i^2}{n} .
\end{eqnarray}

To count the number $m_i(C_4)$ of $4$-cycles in the link graph $L(v_i)$, consider the expression
$$2 \cdot m_i(C_4) = \sum_{\{v_j, v_\ell \} \in [V(L(v_i))]^2} \binom{|N_{L(v_i)}(v_j) \cap N_{L(v_i)}(v_\ell)|}{2},$$
By an averaging argument,  we obtain 
\begin{eqnarray} \label{eq:pairs3}
2 \cdot m_i(C_4) &\geq& \binom{n'}{2} \cdot \binom{\frac{ \sum_{\{v_j, v_\ell \} \in [V(L(v_i))]^2} |N_{L(v_i)}(v_j) \cap N_{L(v_i)}(v_\ell)|}{ \binom{n'}{2}}}{2} \nonumber \\
& = &  \binom{n'}{2} \cdot \binom{\frac{   \sum_{t=1}^{n'}  \binom{  d_i(v_t)}{2} }{ \binom{n'}{2}}}{2} \nonumber \\
& \stackrel{(\ref{eq:pairs1})}{\geq}&  \binom{n'}{2} \cdot \binom{\frac{d_i^2/n }{\binom{n'}{2}} }{2} \nonumber \\
&\stackrel{(d_i \geq n^{3/2}; \, \eqref{realbinomial})}{\geq}&
\frac{d_i^4}{2 n^4}  .
\end{eqnarray}

Inequality~(\ref{eq:pairs3}), together with $\sum_{i=1}^{n'} d_i \geq 3\alpha' n^3$ and the convexity of the function $g(x) = x^4$, leads to
\begin{eqnarray} \label{eq:pairs4}
\sum_{i=1}^{n'} m_i(C_4) \geq \sum_{i=1}^{n'} \frac{d_i^4}{4n^4} \geq \frac{n'}{4n^4} \cdot \left( \frac{\sum_{i=1}^{n'} d_i}{n'} \right)^4 \geq
\frac{1}{4n^4} \cdot \frac{\left(3\alpha' n^3\right)^4}{n'^3} \geq \frac{3^4}{4}  \cdot \alpha'^4 \cdot n^5.
\end{eqnarray}

Notice that the same cycle $C_4$ in the link graphs $L(v_i)$ and $L(v_{j})$, $i \neq j$, gives a copy of $K^{(3)}(2,2,2)$ in $H'$. We may now use an averaging argument to find a lower bound on the number of copies of $K^{(3)}(2,2,2)$ in $H'$ (hence in $H$). If both link graphs $L(v_i)$ and $L(v_j)$ have a copy of a $C_4$ on the vertex set $w_i, w_j, w_k, w_\ell$, the copies must be the \emph{same} in order to produce a copy of $K^{(3)}(2,2,2)$ in $H$. For fixed vertices  $w_i,w_j,w_k,w_\ell$, let $C_4(\{ w_i, w_j; w_k, w_\ell\})$ be the number of link graphs in  $L(v_1),\ldots,L(v_{n'})$ that contain a copy of $C_4$ with vertex set $\{w_i,w_j,w_k,w_\ell\}$ such that the bipartition is $\{w_i,w_j\} \cup \{w_k,w_\ell\}$. As we sum over all possible copies of $C_4$ in the sum below, since every copy of $K^{(3)}(2,2,2)$ is counted three times, we have
\begin{eqnarray*}
&&\frac{1}{3} \sum_{\{i,j,k,\ell\} \in [n']^4} \left( \binom{C_4(\{ w_i, w_j; w_k, w_\ell\})}{2}+\binom{C_4(\{ w_i, w_k; w_j, w_\ell\})}{2}+\binom{C_4(\{ w_i, w_\ell; w_j, w_k)\}}{2} \right)\\
&\geq& \frac{1}{3} \sum_{\{i,j,k,\ell\} \in [n']^4} 3 \binom{\frac{C_4(\{ w_i, w_j; w_k, w_\ell\})+C_4(\{ w_i, w_k; w_j, w_\ell\})+C_4(\{ w_i, w_\ell; w_j, w_k\})}{3}}{2}\\
&\geq&
\frac{1}{3} \cdot 3  \binom{n'}{4} 
\cdot 
\binom{\frac{\sum_{i=1}^{n'} m_i(C_4)} 
{ 3\binom{n'}{4}  }}{2} 
\nonumber \\
&\stackrel{(\ref{eq:pairs4})}{\geq} &  \frac{1}{3} \cdot 3 \binom{n'}{4}
\cdot  \binom{ \frac{ \frac{3^4}{4}  \cdot \alpha'^4 \cdot n^5}{3\binom{n'}{4} }}{2} \nonumber \\
&\stackrel{\eqref{realbinomial}}{\geq}&  \frac{3^7}{2^3} \cdot \alpha'^8 \cdot n^{6} \nonumber \\
&\stackrel{(\alpha' = \alpha/2)}{\geq}&  \frac{3^{7} \cdot\alpha^8 \cdot n^6}{2^{11}}, \nonumber
\end{eqnarray*}
as claimed.
\end{proof}

\section{Proof of Theorem~\ref{main}}

In this section, we prove Theorem~\ref{main}. Our argument follows the steps of~\cite[Theorem 1.2]{keevash2005}, but we adjust the constants in a way that gives us a better dependency between $\delta$ and $\eps$. No attempt was made to find good bounds on $n_0$.

\begin{proof}[Proof of Theorem~\ref{main}]
Throughout the proof we assume that $n$ is sufficiently large and $\delta > 0$  is fixed and sufficiently small. With foresight, we fix $\delta< (1/36)^8$ and, using $c' = 3^7/2^{11}$ from Lemma~\ref{lemma:lemma_2.6}, we also fix

$\delta_1 = (5/3)\delta^{1/2}$   

$\delta_2 = \delta^{1/2}$

$\delta_3 = \delta_1^{1/2}$ 

$\delta_4 =  \delta_1^{1/2}$

$\delta_5 = 6 \delta_3 = 6  \delta_1^{1/2}$ 

$\delta_6 = 18 \delta_1^{1/2}$

$\delta_7 = \sqrt{260} \delta^{1/8}$

$\delta_8 = 90 \delta_1^{1/4}$

$\delta_9 = 417 \delta^{1/8}$

$\delta_{10} = 9 \delta_3 = 9\delta_1^{1/2}$

$\delta_{11} = (27/64)  (834/c')^{1/8} \delta^{1/64}  < 1.94  \delta^{1/64} $.


Let $H$ be a $3$-uniform $\Fano$-free hypergraph on $n\geq n_0$ vertices whose number of hyperedges is at least  $ \ex(n,\Fano) - \delta n^3/8$. If $H$ contains a vertex of degree smaller than $(1 - \delta_1) 3n^2/8 +3n$, then we delete this vertex and continue this process as long as there are vertices of degree less than   $(1 - \delta_1) 3n^2/8 +3n$. If we had deleted
$\delta_2 n -4$ vertices, then we would have arrived at a hypergraph on $(1 - \delta_2)n +4$ vertices with at least
\begin{eqnarray*} \label{eq:-4}
&& \ex(n,\Fano) - \delta n^3/8 - (\delta_2 n - 4)  ((1-\delta_1)3n^2/8 + 3n) \nonumber \\
&\geq& (1 - 2\delta -3 \delta_2(1-\delta_1))  \ex(n,\Fano) + 12(1-\delta_1)n^2/8 - 3\delta_2n^2 + 12n 
\end{eqnarray*}
hyperedges, for $n$ sufficiently large. We claim that, for $n$ sufficiently large,  for the choice $\delta_1 = (5/3)\delta^{1/2}$ and $\delta_2 = \delta^{1/2}$
this is larger than 
$$
\ex((1-\delta_2)n +4, \Fano),
$$
which would imply that there is a Fano plane in $H$, a contradiction. We will show that
\begin{eqnarray} \label{eq:lh1}
1 - 2\delta -3 \delta_2(1-\delta_1) > (1 - \delta_2)^3, 
\end{eqnarray}
which will confirm the claim for $n$ sufficiently large. Now (\ref{eq:lh1}) is equivalent to 
\begin{eqnarray} \label{eq:lh2}
-2\delta + 3 \delta_1 \delta_2 - 3 \delta_2^2 + \delta_2^3 > 0. 
\end{eqnarray}
By our choice of $\delta_1, \delta_2$ by (\ref{eq:lh2}) this is equivalent to
\begin{eqnarray*} \label{eq:lh3}
 \delta^{3/2}  > 0,
\end{eqnarray*}
 as claimed. Let $V_0$ be the subset of $V$ with all vertices that have not been deleted in this process. Hence  all degrees in the subhypergraph $H[V_0]$ induced by $V_0$ are at least $(1-\delta_1)3n^2/8 +3n$. Our discussion implies that
$n_0 = |V_0| \geq n-\delta^{1/2}n +4$. Note that the number of hyperedges of $H$ that are not in $H[V_0]$ is at most $\delta_2 n (1 - \delta_1)  3n^2/8 + 3\delta_2 n^2 \leq  3\delta^{1/2} n^3/8$, for sufficiently large $n$.
Since  $\delta < 1/400$, the subhypergraph $H[V_0]$ contains at least 
$$\ex(n, \Fano)- \delta n^3/8 - 3\delta^{1/2} n^3/8\geq 0.6 \cdot  \binom{n_0}{3}$$
 hyperedges for sufficiently large $n$, and by Lemma~\ref{lemma:lemma_2.2} it contains a copy of the complete hypergraph $K_4^{(3)}$ with vertex set $S=\{a,b,c,d\}$. Let $L(a)$, $L(b)$, $L(c)$ and $L(d)$ be the link graphs of the vertices in $S$. The number of edges in $L(S)$ that are incident with at least one vertex in $S$ is at most $12 n$, as any vertex $x \in V$ is incident with at most three vertices of $S$ in each of the four link graphs.  
 Let $V_1 = V_0 \setminus S$ and $n_1 = |V_1| \geq (1-\delta_2)n$. The link multigraph $G=G(S)$ contains at least $e(L(a)) + e(L(b)) + e(L(c)) + e(L(d)) - 12n  \geq  4[(1-\delta_1) 3 n^2 /8 + 3n] - 12n \geq (1-\delta_1)3 n_1^2/2$ edges. Theorem~\ref{main} is a consequence of the following structural result for the multigraph $G$.
\begin{lemma} \label{lemma:claim}
There is a partition of $S$ into two sets of size two, say $\{a,b\}$ and $\{c,d\}$, for which the following holds. There exists a partition $V_1 = A \cup B $  of the vertex set $V_1$  such that all but at most $\delta_9 n_1^2$ pairs $\{v,w\} \subset V_1$ satisfy the following:
 \begin{itemize}
\item[(i)] Every pair contained in $A$ is an edge in both link graphs $L(a)$ and $L(b)$ but not an edge in $L(c)$ or $L(d)$. 
\item[(ii)] Every pair contained in $B$ is an edge in both link graphs $L(c)$ and $L(d)$ but not in $L(a)$ or $L(b)$.
\item[(iii)] Every crossing pair between $A$ and $B$ is an edge in all four link graphs $L(a)$, $L(b)$, $L(c)$ and $L(d)$.
\end{itemize}
Moreover, the partition satisfies 
$$|A|, |B| \geq \frac{n_1}{4}.$$
\end{lemma}

Before proving Lemma~\ref{lemma:claim}, we show that it implies the desired result Theorem~\ref{main}. 

Let $V_1=A' \cup B'$ denote the partition given by Lemma~\ref{lemma:claim}. Define a partition $A \cup B$ of $V$ by adding each vertex of $V \setminus V_1$ to $A'$ or $B'$ arbitrarily. We choose $\delta_{11}$ satisfying $\delta_{11}= (27/64) \left(2 \delta_9/c'\right)^{1/8},$ where $c'$ is given in Lemma~\ref{lemma:lemma_2.6}.

To obtain the required bound $e(A) + e(B) < 2 \delta_{11} n_{1}^3$, we suppose, for a contradiction, that one of the classes, say $A$, contains at least $\delta_{11} n_{1}^3$ hyperedges of $H$. By the definition of $V_1$, at least
 $$
\delta_{11} n_{1}^3-3 \delta_2 n^3/8 \geq \frac{\delta_{11}n_1^{3} }{2^{\frac{1}{8}}} = \frac{\delta_{11}n_1^{3} }{2^{\frac{1}{8}}|A'|^{3}} \cdot |A'|^3 \stackrel{|A'| \leq (3/4)n_1}{\geq} \frac{64\delta_{11}}{27 \cdot 2^{\frac{1}{8}}} \cdot |A'|^3 
$$
of these hyperedges lie in $A'$. This expression cannot be greater than $\binom{|A'|}{3}$, so that Lemma~\ref{lemma:lemma_2.6} applies and by (\ref{eq:lower9}), $H[A']$ contains at least
$$
\frac{c'\left(\frac{27}{64}\right)^8 \cdot \frac{2\delta_9}{c'} \cdot n_1^{24}}{2|A'|^{24}} \cdot |A'|^6 \stackrel{|A'| \leq (3/4)n_1}{\geq}  \frac{\delta_9 n_1^2 |A'|^4}{16} 
$$ 
copies of $K^{(3)}(2,2,2)$. If we find three vertex-disjoint edges $e_1,e_2,e_3$ in $L(a)$ such that $e_1,e_2,e_3$ form three classes of the partition of some copy of $K^{(3)}(2,2,2)$, then $H$ contains a Fano plane by Lemma \ref{lemma:observation_2.1}~(ii), a contradiction. On the other hand, by Lemma~\ref{lemma:claim}~(i) only at most $\delta_9 n_1^2$ pairs of vertices in $A'$ are not edges of $L(a)$. Clearly, these pairs can be contained in the vertex sets of at most $\delta_9 n_1^2 \binom{|A'|}{4} \leq \delta_9 n_1^2 |A'|^4 /24$ copies of $K^{(3)}(2,2,2)$, thus there is a copy of $K^{(3)}(2,2,2)$ in $H(A')$ such that all three pairs in the partition of its vertex set belong to $L(a)$, yielding a copy of a Fano plane by Lemma~\ref{lemma:observation_2.1}~(ii), which is the desired contradiction.
\end{proof}

\begin{proof}[Proof of Lemma~\ref{lemma:claim}]

We shall construct a sequence $V_1 \supseteq V_{2} \supseteq \cdots \supseteq V_5$ of subsets of the vertex set $V_1$ of $H[V_1]$, where $|V_j| = n_j$, $j = 1,  \ldots, 5$, and $V_1$ is the set containing  $n_1 \geq n-\delta_2 n$ vertices with degree at least $(1 - \delta_1) 3n^2/8$ that has been defined in the first part of the proof.  

For $S=\{a,b,c,d\}$, consider the link multigraph $G(S)$ on $n_1 \geq (1 - \delta_2)n$ vertices with at least $(1-\delta_1) 3n_1^2/2$ edges. If $G$ contains a vertex of degree less than $(1- \delta_3)3n_1$, then we delete this vertex. As before, we continue this deletion process until no vertex satisfies this property. If we had deleted $\delta_4 n_1 $ vertices, then we would have arrived at a multigraph with at least
\begin{eqnarray} \label{eq:lh4}
\left( 1 - \delta_1 - 2 \delta_4 (1 - \delta_3) \right) 3n_1^2/2
\end{eqnarray}
edges. We claim that for $\delta_3 = \delta_4 = \delta_1^{1/2}$ this is at least
$$3 [(1-\delta_4)n_1]^2/2.$$
Indeed, the inequality
\begin{eqnarray*} \label{eq:lh4a}
1 - \delta_1 - 2\delta_4(1 - \delta_3) \geq (1-\delta_4)^2
\end{eqnarray*}
is equivalent to 
\begin{eqnarray} \label{eq:lh5}
-\delta_1 + 2\delta_3 \delta_4  - \delta^2_4   \geq 0,
\end{eqnarray}
which holds by our choice of $\delta_3 = \delta_4 = \delta_1^{1/2}$. By Lemma~\ref{lemma:lemma_2.3}  we find four vertices spanning at least $21$ edges, hence by Lemma~\ref{lemma:lemma_2.4}  we have a Fano plane,  a contradiction. 

Thus, we deleted at most $\delta_4 n_1= \sqrt{(5/3)}\delta^{1/4}n_1$ vertices. This produces a subset $V_2 \subseteq V_1$ with $n_2=|V_2| \geq  n_1- \sqrt{(5/3)}\delta^{1/4}n_1$. We deleted at most $\delta_4 n_1 (1 - \delta_3)3n_1 \leq 3\delta_1^{1/2} n_1^2$ edges from $G$. By construction, all degrees in the subgraph $G[V_2]$ are at least $(1 - \delta_3)3 n_1 \geq (1 - \delta_3)3 n_2$. We shall refer to this inequality as the \emph{degree condition}.


Next, we distinguish two cases according to whether the subgraph $G[V_2]$ contains three vertices spanning at most ten edges or not. As it turns out, the second case will lead to a contradiction.

\bigskip

\noindent \emph{Case~1:}  Suppose that every three vertices in $G[V_2]$ span at most ten edges. In $G[V_2]$ there must be some edge of multiplicity $4$; otherwise, denoting the number of edges of multiplicity $3$ in $G[V_2]$ by $e_3$,  the degree condition leads to
$$
2\left(\binom{n_2}{2} - e_3\right) + 3e_3 \geq  (1- \delta_3) 3n_2^2/2,
$$
which implies  $e_3 \geq (1 - 3\delta_3)n_2^2/2$. Then, by Tur\'an's theorem for $\delta_3 < 1/12$, the subgraph $G[V_2]$
contains a complete graph $K_5$ whose edges have multiplicity at least 3. By Lemma~\ref{lemma:lemma_2.5}~(ii) the hypergraph $H$ must contain a Fano plane, a contradiction. 

Let $\{p,q\}$ be an edge of multiplicity $4$. As we are in case~1, for each vertex $r$ in $G[V_2]$ there are at most six edges between $r$ and $\{p,q\}$. Furthermore, by the degree condition, there are at least $(1-\delta_3)6n_2 - 8$  edges between $\{p,q\}$ and $V_2 - \{p,q\}$ in $G[V_2]$.  

We partition  $V_2 - \{p,q\}$ into four sets $A,B,C,D$ such that for each vertex $x$ in $A$ the edge $\{x,p\}$ has multiplicity $4$ and edge $\{x,q\}$ has multiplicity $2$, for each vertex $x$ in $B$ the edge $\{x,p\}$ has multiplicity $2$ and edge $\{x,q\}$ has multiplicity $4$ and for each $x$ in $C$ both edges $\{x,p\}$ and $\{x,q\}$ have multiplicity $3$. For each vertex  $x \in D$ the sum of the multiplicities of the edges $\{p, x\}$ and $\{q,x\}$ is at most $5$.

In the following we will show that the sizes of both sets $C$ and $D$ are small.

 Let $a_i$ be the number of vertices in $V_2 - \{p,q\}$ that are connected to $\{p,q\}$ via $i$ edges. Note that, because we are in case 1, $i \in\{ 0,\ldots,6\}$.  We have $\sum_{i=0}^{6} a_i = n_2-2$, and
\begin{eqnarray*} \label{delta5_1}
5(n_2 - 2 - a_6) + 6a_6 \geq \sum_{i=0}^6 i \cdot a_i \geq (1-\delta_3) 6n_2 - 8, 
\end{eqnarray*}
which is equivalent to
$$a_6 \geq (1 - 6\delta_3)n_2 + 2.$$
Hence, $\sum_{i=0}^5 a_i \leq 6 \delta_3 n_2 = 6  \sqrt{(5/3)}\delta^{1/4} n_2$, and we set $\delta_5 = 6 \delta_3 = 6 \delta_1^{1/2} = {6\sqrt{(5/3)}}\delta^{1/4}$. 
This implies that the set $D$ in the partition contains at most $\delta_5 n_2$ vertices. Let $V_3=V_2-D$. Note that $n_3=|V_3| \geq (1-\delta_5)n_2$ and that $e(G[V_3])\geq  e(G[V_2])-4\delta_5n_2^2$. 

Now we consider the set $C$. If the subgraph $G[C]$ contains an edge $\{x,y\}$ of multiplicity at least $3$, then $p,q,x,y$ form a $K_4$ in $G[V_3]$ satisfying the hypothesis of Lemma~\ref{lemma:lemma_2.5}~(i), so $H$ must contain a Fano plane, a contradiction. Thus subgraph $G[C]$ contains only edges of multiplicity at most $2$. Moreover, if there is an edge $\{ x, y\}$ with multiplicity at least $2$ between $A \cup B$ and $C$, then $p,q,x,y$ form a $K_4$ as in Lemma~\ref{lemma:lemma_2.5}~(iv), and $H$ must contain a Fano plane, a contradiction. Therefore, all edges between $C$ and $A \cup  B$   have multiplicity at most $1$.
Since we are in case~1, every edge inside $A$ or $B$ has multiplicity at most $2$. Thus, the maximum possible number of edges in $A \cup B$ is achieved when every edge inside $A$ or $B$ has multiplicity $2$ and every edge with one vertex in $A$ and one vertex in $B$ has multiplicity $4$. Therefore, 
\begin{eqnarray} \label{eq:pp1}
&&\frac{3}{2}n_1^2 - 29.5\delta_1^{1/2} n_2^2 \nonumber\\
&\leq& (1 - \delta_1)\frac{3}{2}n_1^2- 3\delta_4 n_1^2 - 4\delta_5 n_2^2 \nonumber\\
&\leq& e(G[V_3]) \nonumber\\
&\leq& 2 \binom{|A|}{2} + 2 \binom{|B|}{2} + 4|A||B| + 2 \binom{|C|}{2} + |C| |A \cup B|.
\end{eqnarray}
For fixed size $|A| + |B|$, by taking the derivative,  (\ref{eq:pp1}) is maximum for $|A|=|B|=m$. Then, with $|C| = n_3 - 2m$, the term~(\ref{eq:pp1}) is at most 
$$
6m^2 -2m +2m|C| + |C|^2 - |C|= 6m^2 -2mn_3 + n_3^2 - n_3 \leq 6m^2 -2mn_1 + n_1^2,
$$
and therefore with~(\ref{eq:pp1})
$$
\frac{3n_1^2}{2} - 29.5 \delta_1^{1/2} n_1^2 \leq 6m^2 -2mn_1 + n_1^2,
$$
thus, for $n$ sufficiently large, 
$$
m \geq \frac{n_1}{6} + \frac{n_1}{3} \cdot \sqrt{1 - 45 \delta_1^{1/2}} + 1.
$$
Using $\sqrt{1-x} \geq 1-3x/5$ for $0 \leq x \leq 5/9$, with $\delta_1^{1/2} \leq 1/81$, we infer 
$$
m \geq \frac{n_1}{2} (1 - 18 \delta_1^{1/2}) +1,
$$
thus $|C| \leq 18\delta_1^{1/2}n_1$.  Let $\delta_6 = 18\delta_1^{1/2}$. We produce $V_4$ by deleting all vertices in $C$, which leads to the deletion of at most $4 \delta_6 n_1^2$ edges incident with them.  For the sake of simplicity, we also delete $p$ and $q$ from $G[V_3]$. Therefore $G[V_4]$ has $n_4 \geq n_3- 18\delta_1^{1/2}n_1$ vertices.

Now we show that $A$ and $B$ satisfy the conditions of the lemma (we note that, near the end of the proof, each vertex in $C \cup D \cup (V\setminus V_1)$ will be added arbitrarily to $A$ or $B$).  Note that to obtain $G[V_4]$, at most 
\begin{eqnarray} \label{n1n4edges}
(3\delta_1^{1/2} + 4 \delta_5 + 4 \delta_6)n_1^2 \leq 128\delta^{1/4} n_1^2
\end{eqnarray} 
edges have been removed from $G = G[V_1]$.

Let $E'$ be the set of all pairs of vertices within $A$ or $B$ of multiplicity at most $1$ and pairs of vertices between $A$ and $B$ of multiplicity at most $3$.
Since $|A| + |B| = n_4$, we deduce
\begin{eqnarray*}
(1-\delta_1)3 \frac{n_1^2}{2} &\leq& e(G[V_1]) \leq 2 \binom{|A|}{2} + 2 \binom{|B|}{2} + 4|A||B| + 128 \delta^{1/4}n_1^2- |E'| \\
&\leq& 3 \binom{|A| + |B|}{2} - \frac{(|A| - |B|)^2}{2} + 128 \delta^{1/4}n_1^2 - |E'| \\
&\leq& 3 \binom{n_4}{2} - \frac{(|A| - |B|)^2}{2} + 128\delta^{1/4}n_1^2 - |E'| \\
&\leq& 3 \binom{n_1}{2} - \frac{(|A| - |B|)^2}{2} + 128\delta^{1/4}n_1^2 - |E'|,
\end{eqnarray*}
or, for $\delta < 4^4/5^4$,  
\begin{eqnarray*}
\frac{(|A| - |B|)^2}{2} + |E'| &\leq& \left(\frac{3}{2} \delta_1 + 128 \delta^{1/4}\right) n_1^2 \leq 130\delta^{1/4} n_1^2.
\end{eqnarray*}
Therefore,  we infer $|E'| \leq 130 \delta^{1/4} n_1^2$ and with  $\delta_7 = \sqrt{260} \delta^{1/8}$  also
$$\big||A| - |B|\big| \leq \delta_7 n_1,$$ which, combined with $|A| + |B| = n_4$, implies 
\begin{equation}\label{sizeA}
\frac{n_4}{2} - \frac{\delta_7 n_1}{2} \leq |A|, |B| \leq  \frac{n_4}{2} + \frac{\delta_7 n_1}{2}.
\end{equation}
To find a lower bound on $|A|$ in terms of $n_1$, we first note that
\begin{eqnarray*}
n_4 &\geq& n_3 - 18 \delta_1^{1/2} n_1 \\
&\geq& (1-\delta_5) n_2 - 18 \delta_1^{1/2} n_1 \\
&\geq& (1-\delta_5) [n_1 - \delta_1^{1/2}n_1] - 18 \delta_1^{1/2} n_1 \\
&\geq& n_1 - 25\delta_1^{1/2}n_1. 
\end{eqnarray*}
With~\eqref{sizeA}, this leads to
\begin{eqnarray} \label{eq:lower7}
\nonumber |A| &\geq& \frac{n_4}{2} - \frac{\sqrt{65}}{2} \delta^{1/8} n_1\nonumber \\
&\geq& \frac{n_1}{2} -\frac{25}{2}\delta_1^{1/2}n_1 - \sqrt{260} \delta^{1/8} n_1/2  \nonumber\\
&=& \frac{n_1}{2} - \left(\frac{25}{2} \sqrt{\frac{5}{3}} \delta^{1/8} + \sqrt{65} \right) \delta^{1/8} n_1 \nonumber \\
&\geq& \frac{n_1}{2} - 9 \delta^{1/8}n_1. 
\end{eqnarray}
We used that $\delta<(1/36)^{8}$, which also leads to
\begin{eqnarray} \label{eq:lower8}
 |A| \geq  \frac{n_1}{2} - 9 \delta^{1/8}n_1 \geq  \frac{n_1}{4}.
 \end{eqnarray}
The same applies to $|B|$. Then, since $|A| + |B| \leq n_1$, inequality~(\ref{eq:lower8}) implies
\begin{eqnarray} \label{eq:lower9}
 \frac{n_1}{4} \leq |A|, |B| \leq \frac{3n_1}{4} .
\end{eqnarray}

By \eqref{n1n4edges}, 
for all but at most
\begin{eqnarray} \label{notyetdelta9}
|E'| + 128 \delta^{1/4} n_1^2  \leq 258 \delta^{1/4} n_1^2
\end{eqnarray}
pairs $\{x,y\}$ of vertices $x,y \in V_1$, the multiplicity of $\{x,y\}$ in $G[V_4]$ is equal to $2$, if $\{x,y\} \subset A$ or $\{x,y\} \subset B$, or equal to $4$, if $x \in A$ and $y \in B$. It remains to show that  most pairs $\{x,y\} \subset A$ are edges in the same two link graphs and most pairs $\{x,y\} \subset B$ are edges in  the two other link graphs. 

Let $d(x)$ be the degree of vertex $x$ in $G[V_4]$. 
For a vertex $x \in A$, let $B(x) \subseteq B$ be the set of vertices in $B$ joined to $x$ by an edge of multiplicity $4$. For a vertex $u \in B$, let $A(u) \subseteq A$ be defined correspondingly. Since $|A|$ and $|B|$ are each at most $n_4/2 + \delta_7n_1$ and at least $n_4/2 - \delta_7n_1$, and their sum is $n_4$, we have
\begin{eqnarray*}
(1-\delta_3)3n_2 - 4(\delta_5 + \delta_6)n_1 &<& d(x) \hspace{1cm}  \\
&<& 2 |A| + 4|B| - (|B| - |B(x)|) \\
&\leq & 2(n_4/2 - \delta_7n_1/2) + 4(n_4/2 + \delta_7n_1/2) - (|B| - |B(x)|)\\
&=& 3 n_4 + \delta_7 n_1 - (|B| - |B(x)|),
\end{eqnarray*}
so
\begin{eqnarray*}|B| - |B(x)| &\leq& (3\delta_3 + 4\delta_5 + 4\delta_6 + \delta_7)n_1\\
&\leq& (3 \delta_1^{1/2} + 24 \delta_1^{1/2} + 72 \delta_1^{1/4} +  17 \delta_1^{1/4})n_1 \\
&\leq& (27 \delta_1^{1/4} + 89) \delta_1^{1/4}n_1 \leq 90  \delta_1^{1/4}n_1,
\end{eqnarray*}
 and hence $|B| - |B(x)| < \delta_8 n_1$, where $\delta_8 = 90\delta_1^{1/4}$.  This means that every vertex $x \in V_4$ is incident with at most $\delta_8 n_1$ vertices of the other side with edges that do not have multiplicity 4.

Now fix an edge $\{x,y\}$ of multiplicity $2$ inside $A$. Without loss of generality assume that it lies in the edge set of both $L_A(a)$ and $L_A(b)$, where set subscripts indicate subgraphs of the link graph induced by the corresponding set. At this point, we would know that most almost all pairs of vertices in $A$ are edges of multiplicity $2$ and that the same holds for $B$. Since $|B| - |B(x) \cap B(y)| < 2 \delta_8 n_1$, we can delete all the vertices in $B - (B(x) \cap B(y))$ and assume that $B = B(x) \cap B(y)$.  Now no edge $\{w,z\}$ of $B$ can be in $L_B(a)$, as then partitioning the edges of a copy of the complete graph $K_4$ on the vertex set  $\{w,x,y,z\}$ as $M_a = \{\{x,y\},\{w,z\}\}$, $M_b = \{\{w,x\},\{y,z\}\}$, $M_c = \{\{w,y\},\{x,z\}\}$ gives a Fano plane  by Lemma~\ref{lemma:observation_2.1}~(i). Similarly no edge in $B$ can be in $L_B(b)$. 

Now let $\{u,v\}$ be an edge of multiplicity $2$ in $B$. Then, as we just proved, it lies in the edge set of both $L_B(c)$ and $L_B(d)$. So arguing as above, and deleting the at most $2 \delta_8 n_1$ vertices in $A - (A(u) \cap A(v))$, we can assume that all vertices in $A$ are adjacent to both $u$ and $v$ by edges of multiplicity $4$, and therefore no edge in $A$ can be in $L(c)$ or $L(d)$. More formally, the above vertex-deletions produce a graph $G_5$ on $n_5 \geq n_4 -  4 \delta_8 n_1$ vertices. Now we distribute all the vertices of $V_1$ that have been deleted up to this point to make $A \cup B$ a partition of $V_1$. At this point we find the upper bound $\delta_9 n_1^2$ on the number of pairs of vertices in $V_1$ that do not satisfy the conditions of Lemma~\ref{lemma:claim} by summing the upper bounds \eqref{notyetdelta9} and $4\delta_8 n_1^2$  on  the number of pairs deleted from $G$, that is,
\begin{eqnarray*}
258 \delta^{1/4} n_1^2 + 4 \delta_8 n_1^2 
&\leq& 258 \delta^{1/4} n_1^2 + 4 \cdot 90\delta_1^{1/4} n_1^2 \\
&\leq& \left(258\delta^{1/8}  + 360 (5/3)^{1/4} \right) \delta^{1/8} n_1^2 \\
&\leq& 417\delta^{1/8} n_1^2,
\end{eqnarray*} 
and we set $\delta_9 = 417\delta^{1/8}$. 
This shows that $A \cup B$ is the required partition, and completes the analysis of Case~1.
\vspace{0.4cm}

\noindent \emph{Case~2:} Suppose next that there exist three vertices $p,q,r$ in $G[V_2]$ that span at least $11$ edges. Without loss of generality we can assume that $\{p,q\}$ and $\{p,r\}$ have multiplicity $4$ and $\{q,r\}$ has multiplicity at least $3$. By the degree condition, there are at least $(1 - \delta_3)9n_2 - 24$  edges between $\{p,q,r\}$ and $V_2 - \{p,q,r\}$.

 By Lemma~\ref{lemma:lemma_2.4}  for each vertex $s$ the sum of the multiplicities of all edges from $s \in V_2 - \{p,q,r\}$ to $\{p,q,r\}$ is at most $9$.
Let $a_i$ be the number of vertices in $V_2 - \{p,q,r\}$ that are connected to $\{p,q,r\}$ via $i$ edges, $i = 0,\ldots,9$.  We have $\sum_{i=0}^9 a_i = n_2-3$, and
\begin{eqnarray*} \label{delta7_1}
8(n_2 - 3 - a_9) + 9 a_9 \geq \sum_{i=0}^9 i \cdot a_i \geq (1-\delta_3) 9n_2 - 24,
\end{eqnarray*}
which implies
$a_9 \geq (1 - 9\delta_3)n_2$, and
 $\sum_{i=0}^8a_i \leq 9 \delta_3 n_2 = 9 \delta_1^{1/2} n_2$. We set $\delta_{10} = 9 \delta_1^{1/2}$ and we produce $V_3$ by deleting the at most $\delta_{10} n_2$ vertices from $V_2 - \{p,q,r\}$ that are each connected to $\{p,q,r\}$ by less than nine edges. These vertices  altogether are incident with at most $4 \delta_{10} n_2^2$ edges. Therefore, $G[V_3]$ has $n_3 \geq n_2 - \delta_{10} n_2$ vertices.

By Lemma~\ref{lemma:lemma_2.5}~(i) and (iii) no vertex $s \in V_3 - \{p,q,r\}$ can be adjacent to vertices in $\{p,q,r\}$ by edges with multiplicities $3,3,3$ or $4,3,2$.  Since $V_3$ only contains vertices counted by $a_9$, the degree pattern in $\{p,q,r\}$ of all vertices  $s \in V_3 - \{p,q,r\}$ is $4,4,1$, in particular, the multiplicities of the edges $\{p, s \}$, $\{q,  s\}$ and $\{r, s\}$ are $1,4,4$ in this order, as otherwise,  by Lemma~\ref{lemma:lemma_2.5}~(v), $H$ would contain a copy of the Fano plane.
Suppose that there is some edge $\{s,t\}$ in $G[V_3-\{p,q,r\}]$ of multiplicity at least $2$. Then the vertices $q,r,s,t$ span at least $21$ edges,  and by Lemma~\ref{lemma:lemma_2.4}, $H$ contains a copy of the Fano plane. Thus, all edges in  $G[V_3-\{p,q,r\}]$ have multiplicity at most $1$. On the other hand, by the degree condition, the subgraph $G[V_3-\{p,q,r\}]$ contains at least 
\begin{eqnarray*}
(1-\delta_3)3n_2^2/2 - 4\delta_{10}n_2^2 - 12  n_2 &\geq& 3n_2^2/2 - 38 \delta_1^{1/2} n_2^2 \\ 
&\geq& (3/2 - 38\delta_1^{1/2}) n_3^2
\end{eqnarray*}
edges. For   $\delta_1 < (1/38)^2$, this is not possible. 
Therefore, case~2 always leads to a contradiction. This completes the proof of Lemma~\ref{lemma:claim}.
\end{proof}

\end{document}